\documentclass[onefignum,onetabnum]{siamonline220329}

\ifpdf
\hypersetup{
  pdftitle={Multi-marginal Approximation of the Linear Gromov--Wasserstein Distance},
  pdfauthor={F. Beier, R. Beinert}
}
\fi

\usepackage{graphicx}

\usepackage{amssymb}
\usepackage{stmaryrd}
\usepackage{bm}
\usepackage{mathtools}
\usepackage[shortlabels]{enumitem}
\usepackage{amsmath}

\DeclareMathOperator{\MGW}{MGW}
\DeclareMathOperator{\LGW}{LGW}
\DeclareMathOperator{\GW}{GW}
\newcommand{\eps}{\varepsilon}

\newcommand{\id}{\mathrm{id}}
\newcommand{\p}{\mathcal{P}}
\newcommand{\Pio}{\Pi_\opt}
\newcommand{\opt}{\mathrm{o}}
\newcommand{\vX}{\bm{X}}
\newcommand{\vx}{\bm{x}}
\newcommand{\XX}{\mathbb{X}}
\newcommand{\YY}{\mathbb{Y}}
\renewcommand{\SS}{\mathbb{S}}
\newcommand{\NN}{\mathbb{N}}
\newcommand{\weakly}{\rightharpoonup}\normalfont
\newcommand{\dx}{\,\mathrm{d}}
\DeclareMathOperator{\Tan}{Tan}
\DeclareMathOperator*{\argmin}{argmin}

\newsiamremark{remark}{Remark}
\newsiamremark{example}{Example}

\headers{Multi-marginal Approximation of the Linear Gromov--Wasserstein Distance}
{F. Beier, R. Beinert}

\title{Multi-marginal Approximation of the Linear Gromov--Wasserstein Distance%
  \thanks{This work was funded by
      the German Research Foundation (DFG)
      within the RTG~2433 DAEDALUS
      and by the BMBF project ``VI-Screen'' (13N15754).}}

\author{Florian Beier%
  \thanks{Institute of Mathematics,
    Technische Universit\"at Berlin,
    Stra\ss{}e des 17. Juni 136,
    10623 Berlin, Germany 
    (\email{f.beier@tu-berlin.de},
    \email{robert.beinert@tu-berlin.de}).}
  \and Robert Beinert\footnotemark[2]}

\begin{document}

\maketitle

\begin{abstract}
Recently, two concepts from optimal transport theory have successfully been brought to the Gromov--Wasserstein (GW) setting.
This introduces a linear version of the GW distance and multi-marginal GW transport. 
The former can reduce the computational complexity when computing all GW distances of a large set of inputs. 
The latter allows for a simultaneous matching of more than two marginals, which can for example be used to compute GW barycenters.
The aim of this paper is to show an approximation result which characterizes the linear version as a limit of a multi-marginal GW formulation.
\end{abstract}


\section{Introduction}
A fruitful line of work in the realm of optimal transport are Gromov--Wasserstein (GW) distances \cite{memoli2011gromov},
which allow meaningful comparisons and matchings of probability measure living on different metric spaces.
The main idea is to transport mass such that distances are preserved.
In particular,
GW distances are invariant under isometric transformations,
which makes them an invaluable tool for shape matching and comparison \cite{memoli2011gromov}.
Other applications include e.g.\ dictionary learning \cite{VVFCC21} and graph prediction \cite{pmlr-v162-brogat-motte22a}. 
Since GW amounts to solving a quadratic problem,
which is computationally challenging, 
effort has been made to provide computationally tractable generalizations. 
Examples are sliced GW \cite{sliced_gw}, quantized GW \cite{quantized_gw}, and sampled GW \cite{sampled_gw}. 

In \cite{BBS2021linear},
the authors propose a linear Gromov--Wasserstein distance (LGW),
which provides a significant speed-up
when requiring all pairwise distances of a whole set of inputs
such as in classification tasks.
The motivation of LGW is based on linear optimal transport (LOT),
which was introduced by Wang et al.\ \cite{wang2013linear}. 
Since its introduction, LOT has been successfully applied
for several tasks in nuclear structure-based pathology
\cite{WOSSR2011}, parametric signal estimation \cite{RHNHR2020},
signal and image classification \cite{KTOR2016,PKKR2017,moosmueller2021linear,khurana2022supervised}, modeling of
turbulences \cite{EN2020}, cancer detection
\cite{BKR2014,OTWCKBHR2014,TYKSR2015}, Alzheimer disease detection
\cite{PNAS21}, vehicle-type recognition \cite{GLDY2019} as well as for
de-multiplexing vortex modes in optical communications
\cite{PCNWDR2018}.
Both LOT and LGW make use of the geometrical structure of the (Gromov--)Wasserstein space
and compute distances in the tangent space with respect to some a priori fixed reference.

In \cite{BBS2022multi}, the authors propose multi-marginal Gromov--Wasserstein transport to simultaneously match multiple inputs in the GW sense. Multi-marginal contexts have previously been considered in the OT setting since several decades \cite{GS98}. Most prominently, the multi-marginal formulation may be leveraged to express OT barycenter problems which were introduced in the celebrated work \cite{AC2011}.
Independently of barycenters, multi-marginal OT is pivotal in other tasks such as e.g.\ matching for teams \cite{CE10}, 
particle tracking \cite{CK18} and information fusion \cite{DC14,EHJK20}.

In this paper, we briefly go over the cornerstones related to the linear and multi-marginal version of GW. Additionally, we show an approximation result which characterizes the linear GW distance as a limit of a multi-marginal GW problems.

\section{Linear Gromov--Wasserstein}

In this section we go over fundamental definitions in relation to the Gromov--Wasserstein distance. Furthermore, after giving a brief discussion regarding the derivation of the linear version which is following a similar strategy to the derivation of LOT, we conclude the section with some examples.

A \emph{metric measure space (mm-space)} is a triple $\XX = (X, d_{X}, \mu)$, where
\begin{enumerate}[(i)]
\item $(X, d_{X})$ is a compact metric space,
\item $\mu$ is a probability measure on the Borel $\sigma$-algebra on $(X, d_X)$ with full
  support.
\end{enumerate}
The set of all Borel probability measures on $(X, d_X)$ is denoted by $\p (X)$.
For two mm-spaces  $\XX = (X,d_{X},\mu)$ and $\YY =
(Y,d_{Y}, \nu)$, 
the \emph{Gromov--Wasserstein (GW) distance} is defined by
\begin{align}
&\GW(\XX,\YY) 
\coloneqq
\inf_{\pi \in \Pi(\mu,\nu)} 
\biggl( \int_{(X \times Y)^2} \lvert d_{X}(x,x') - d_{Y}(y,y') \rvert^2 \dx \pi(x,y) \dx \pi(x',y') \biggr)^{\tfrac12}.
    \label{eq:GW}
\end{align}
Here $\pi \in \Pi(\mu,\nu)$ means that $\pi  \in \mathcal P(X\times Y)$
has marginals $\mu$ and $\nu$. Due to the
Weierstraß theorem, a minimizer in \cref{eq:GW} always exists
\cite[Cor~10.1]{memoli2011gromov}. We denote by $\Pio(\XX,\YY)$ the set of all minimizers. Two mm-spaces
$\XX = (X,d_{X},\mu)$ and
$\YY = (Y,d_{Y},\nu)$ are called \emph{isomorphic} if and
only if there exists a measure-preserving isometry $I: X \to Y$ meaning $I_\# \mu \coloneqq \mu \circ I^{-1}= \nu$. 
The isometry classes are henceforth denoted by $\llbracket \cdot \rrbracket$.
The space of all isometry classes of mm-spaces is called the \emph{Gromov--Wasserstein space},
and $\GW$ defines a metric on it. In the Euclidean setting, measure-preserving isometries are characterized by rotations, translations, and reflections,
which makes $\GW$ an invaluable tool in shape and image analysis.

The quadratic dependence on the objective makes the optimization problem in \cref{eq:GW} computationally challenging. This worsens if all pairwise distances of a whole set of mm-spaces are required. To alleviate this issue, recently, a linearized version of $\GW$ has been proposed in \cite{BBS2021linear}.
With respect to some reference mm-space $\SS = (S,d_S,\sigma)$, the \emph{linear Gromov--Wasserstein} distance between $\XX$ and $\YY$ is given by
\begin{equation}\label{eq:LGW}
\LGW_{\SS}(\XX,\YY)
= \inf_{
\substack{\pi \in \p(S \times X \times Y) \\ (P_{S \times X})_\# \pi \in \Pio(\SS,\XX) \\ 
(P_{S \times Y})_\# \pi \in \Pio(\SS,\YY)}}
\int_{(S \times X \times Y)^2}
\lvert d_X(x,x') - d_Y(y,y') \rvert^2\dx \pi(s,x,y) \dx \pi(s',x',y'),
\end{equation}
where $P_{\bullet}$ denotes the projection into the indexed subspace.
An extensive justification and further discussions of $\LGW$ are presented in \cite{BBS2021linear}.
In the following we present the main idea.
To approximate the GW distance, 
we exploit 
that the GW space is \emph{geodesic} \cite{sturm2020space}. 
That is, between every two mm-spaces $\SS$ and $\XX$, there exists a \emph{geodesic} (minimal length path) $t \mapsto \pi_t^{\SS \to \XX}$,
$t \in [0,1]$,
connecting $\SS$ and $\XX$.
Furthermore, there is a one-to-one identification between geodesics and optimal GW plans $\Pio(\SS,\XX)$.
More precisely,
every geodesic is of the form
\begin{equation*}
    t \mapsto \pi_t^{\SS \to \XX}
    \coloneqq
    \llbracket S \times X, (1-t) \, d_S + t \, d_X, \pi \rrbracket,
    \qquad
    t \in [0,1],
\end{equation*}
where $\pi \in \Pio(\SS, \XX)$, 
and where the metric between $(s,x)$ and $(s',x')$ is given by
$(1-t) \, d_S(s,s') + t \, d_X(x,x')$.
On the basis of these geodesics,
and embedding the space of mm-spaces into the space of so-called \emph{gauged measure spaces},
Sturm \cite{sturm2020space} has constructed a tangent space $\Tan_\SS$ at $\SS$, 
which especially contains all geodesics starting in $\SS$.
The constructed tangent space has again a metric structure.
For two geodesics $\pi_t^{\SS \to \XX}$ and $\pi_t^{\SS \to \YY}$ 
related to $\pi_\SS^\XX \in \Pio(\SS,\XX)$ and $\pi_\SS^\YY \in \Pio(\SS,\YY)$,
the metric on $\Tan_\SS$ between these geodesics
is given by
\begin{equation}
\label{eq:GW_S}
\inf_{
\substack{\pi \in \p(S \times X \times Y) \\ (P_{S \times X})_\# \pi = \pi_\SS^\XX \\ 
(P_{S \times Y})_\# \pi = \pi_\SS^\YY}}
\int_{(S \times X \times Y)^2}
\lvert d_X(x,x') - d_Y(y,y') \rvert^2\dx \pi(s,x,y) \dx \pi(s',x',y'),
\end{equation}
see \cite[Prop~III.1]{BBS2021linear}.
Figuratively, 
LGW in \cref{eq:LGW} can consequently be interpreted as the minimal distance between all geodesics from $\SS$ to $\XX$ and all geodesics from $\SS$ to $\YY$.

Notably, if the plans are concentrated on graphs, i.e.\ there exists $T_1 \colon \SS \to \XX$, $T_2 \colon \SS \to \YY$ so that $\pi_\SS^\XX = (\id,T_1)_\# \sigma$ and $\pi_\SS^\YY = (\id,T_2)_\# \sigma$, then the functional in \cref{eq:GW_S} simplifies to
\[
\int_{S^2} \lvert d_X(T_1(s),T_1(s')) - d_Y(T_2(s),T_2(s')) \rvert^2 \dx \sigma(s) \dx \sigma(s').
\]
This observation is the main ingredient for an approximation of $\LGW$, which constructs suitable plans using barycentric projections
and allow the efficient approximation of pairwise GW distances for large sets of mm-spaces,
see \cite{BBS2021linear} for more details.
The most crucial point of LGW is the selection of the reference space $\SS$,
which significantly affects the approximation properties.
Since LGW admits the bounds
\begin{equation}
    \label{eq:LGW_estimates}
    \GW(\XX,\YY) \leq \LGW_\SS(\XX,\YY) \leq \GW(\SS,\XX) + \GW(\SS,\YY),
\end{equation}
see \cite[Lem~III.2]{BBS2021linear},
the reference space should optimally lie between the considered mm-spaces.
If we want to compute the pairwise LGW distances between the mm-spaces $\XX_i$, $i=1, \dots, N$,
a reasonable $\SS$ would be 
\begin{equation*}
    \SS \in \argmin_{\tilde \SS}\sum_{i=1}^N \GW(\tilde \SS, \XX_i),
\end{equation*} 
where the minimization is taken over all mm-spaces.
Since the computation of the minimizer is intractable,
an alternative is given by the \emph{GW barycenter},
which is defined as
\begin{equation}
    \label{eq:bary}
    \SS \in 
    \argmin_{\tilde \SS} 
    \sum_{i=1}^N \GW^2(\tilde \SS, \XX_i).
\end{equation}
The computational aspects of GW barycenters are discussed in \cite{PCS2016,BBS2022multi}.

For certain reference spaces, 
LGW coincides with GW.

\begin{example}
    Let $\SS$ be the single point mm-space,
    i.e.\ $S = \{s_0\}$, $d_S(s_0,s_0) = 0$, 
    and $\sigma = \delta_{\{s_0\}}$. 
    In this case, 
    the set of optimal plans simplifies to
    $\Pio(\SS,\XX) = \{\delta_{s_0} \otimes \mu\}$
    and $\Pio(\SS,\YY) = \{\delta_{s_0} \otimes \nu\}$. 
    The three-plans in \cref{eq:LGW} hence become
    \begin{align*}
        &\{\pi \in \p(S \times X \times Y)
        : (P_{S \times X})_\# \pi \in \Pio(\SS,\XX),
        (P_{S \times Y})_\# \pi \in \Pio(\SS,\YY)\} 
        \\
        &=
        \{\pi \in \p(S \times X \times Y)
        : (P_{S \times X})_\# \pi = \delta_{s_0} \otimes \mu,
        (P_{S \times Y})_\# \pi = \delta_{s_0} \otimes \nu \}
        \\
        &= \{ \delta_{s_0} \otimes \tilde \pi 
        : \tilde \pi \in \Pi(\mu,\nu)\}.
    \end{align*}
    The objective in \cref{eq:LGW} for a feasible $\pi$ thus 
    coincides with the objective in \cref{eq:GW} at $(P_{X\times Y})_\# \pi$;
    so 
    \begin{equation*}
        \LGW_\SS(\XX,\YY) = \GW(\XX,\YY).    
    \end{equation*}
\end{example}

\begin{example}
    Notably,
    the identity $\LGW_\SS(\XX,\YY) = \GW(\XX,\YY)$ 
    also holds for $\SS = \XX$ and $\SS = \YY$
    due to the bounds in \cref{eq:LGW_estimates}.
    In the case $\SS = \XX$ for instance,
    the three-plans in \cref{eq:LGW}
    have the form $(P_X, I, P_Y)_\# \tilde \pi$ 
    with $\tilde \pi \in \Pi(\mu,\nu) \subset \p(X \times Y)$,
    where $I \colon X \to X$ is some measure-preserving isometry.
    Therefore, the values of \cref{eq:GW} and \cref{eq:LGW}
    coincide for $(P_X, I, P_Y)_\# \tilde \pi$ and $\tilde \pi$ respectively.
\end{example}

\section{Multi-marginal Approximation}

Recently, GW was extended to the multi-marginal context \cite{BBS2022multi}. In this section, we show that a multi-marginal formulation can be used to approximate $\LGW$ in \cref{eq:LGW}. Before stating the result, we give a brief discussion of the multi-marginal setting. 
Let
$\XX_i \coloneqq (X_i,d_i,\mu_i)$, $i=1,\dots, N$, be mm-spaces, we set
\begin{equation*}
    \vX \coloneqq \bigtimes_{i=1}^N X_i
    \qquad\text{with}\qquad
    \vx \coloneqq (x_1, \dots, x_N),
    \;
    x_i \in X_i.
\end{equation*}
The \emph{multi-marginal GW transport} problem is given by
\begin{equation}
    \label{eq:MGW}
  \MGW(\XX_1,\dotsc,\XX_N)
  \coloneqq \inf_{\pi \in \Pi(\mu_1,\dots,\mu_N)}
  \int_{\vX^2} c(\vx,\vx') \dx \pi(\vx) \dx \pi(\vx'),
\end{equation}
where $\Pi(\mu_1,\dots,\mu_N) \coloneqq \{\pi \in \p(X) : (P_{X_i})_\# \pi = \mu_i \}$
and $c \colon \vX \times \vX \to [0,\infty]$ is some cost function.
Typically, cost functions $c$ involve the pairwise quadratic distances of the metrics $d_1,\dotsc,d_N$ like in
\begin{equation}   \label{eq:cost}
  c(\vx,\vx')
  \coloneqq
  \sum_{i,j=1}^N
  c_{ij} |d_i(x_i,x'_i) - d_j(x_j, x_j')|^2
  \qquad\text{with}\qquad
  c_{ij} \ge 0.
\end{equation}
Computationally, it is advantageous to solve an entropic, bi-convex relaxation of \cref{eq:MGW} by an alternating minimization scheme \cite{SejViaPey21,BBS2022multi}. Each step of this scheme consists of the minimization of a regularized multi-marginal optimal transport problem, which can be solved efficiently for certain cost functions using the multi-marginal Sinkhorn algorithm from \cite{BLNS2021}.

The multi-marginal $\GW$ methods can be leveraged to compute $\GW$ barycenters.
For arbitrary weights $\rho_i > 0$, $i = 1,\dotsc,N$ with $\sum_{i=1}^N \rho_i = 1$, 
a \emph{free-support barycenter} between $\XX_1, \dots, \XX_N$ 
is defined as
\begin{equation*}
    \SS \in 
    \argmin_{\tilde \SS} 
    \sum_{i=1}^N \rho_i \GW^2(\tilde \SS, \XX_i),
\end{equation*}
where the minimization is taken over all mm-spaces.

\begin{theorem}[Free-Support Barycenter, {\cite[Thm~6.1]{BBS2022multi}}] 
  \label{thm:gen-bary}
  Let $\XX_i$ be given mm-spaces,
  and let $\rho_i \ge 0$ be weights with $\sum_{i=1}^N \rho_i = 1$. 
  Then a free-support barycenter $\SS^* \coloneqq (S^*, d^*, \sigma^*)$
  is given by
  $S^* \coloneqq \vX$,
  $d^*(\vx,\vx') \coloneqq \sum_{i=1}^N \rho_i d_i(x_i,x_i')$,
  and $\sigma^* \coloneqq \pi^*$,
  where $\pi^*$ is a minimizer of $\MGW(\XX_1, \dots, \XX_N)$
  with cost function
  \begin{equation*}
    c(\vx,\vx')
    \coloneqq
    \frac12
    \sum_{i,j=1}^N
    \rho_i \rho_j \,
    \lvert d_i(x_i,x_i') - d_j(x_j,x_j') \rvert^2.
  \end{equation*}
\end{theorem}
In certain cases it is desirable to fix the support of the barycenter beforehand. This leads to the \emph{fixed-support barycenter formulation}
\begin{equation*}
  \argmin_{\sigma \in \p(S)}
  \sum_{i=1}^N \rho_i \GW^2(\XX_i,\SS)
  \qquad\text{with}\qquad
  \SS = (S, d, \sigma).
\end{equation*}

\begin{theorem}[Fixed-Support Barycenter, {\cite[Thm~6.2]{BBS2022multi}}]
  \label{thm:rest-bary}
  Let $\XX_i$ be given mm-spaces,
  and let $\rho_i \ge 0$ be weights
  with $\sum_{i=1}^N \rho_i = 1$.
  A fixed-support barycenter
  $\SS^* \coloneqq (S, d, \sigma^*)$
  is given by $\sigma^* \coloneqq (P_S)_\# \pi^*$,
  where $\pi^*$ minimizes
  \begin{equation}
  \label{eq:fix-supp-bary}
  \inf_{\substack{\pi \in \p(X_1 \times \cdots \times X_N \times S)\\ 
  (P_{\vX})_\# \pi \in \Pi(\mu_1,\dots,\mu_N)}}
  \int_{(\vX \times S)2} c((\vx,s),(\vx',s')) \dx \pi(\vx,s) \dx \pi(\vx',s'),
  \end{equation}
  with cost function
  \begin{equation*}
    c((\vx,s),(\vx',s'))
    \coloneqq
    \sum_{i=1}^N \rho_i
    | d_i(x_i,x_i') - d(s,s')|^2.
  \end{equation*}
\end{theorem}

\begin{remark}
    The minimization problem in \cref{eq:fix-supp-bary},
    where the last marginal of $\pi \in \p(\vX \times S)$ is unconstrained,
    is an instance of \emph{unbalanced} multi-marginal GW, see \cite{BBS2022multi}.
\end{remark}

For the remainder of the section, let $\SS = (S,d_S,\sigma)$, $\XX = (X,d_X,\mu)$ and $\YY = (Y,d_Y,\nu)$ be mm-spaces. For $\eps > 0$, let
\begin{align*}
    &c_{\eps}((s,x,y),(s',x',y')) 
    \\
    &\coloneqq \eps (d_X(x,x') - d_Y(y,y'))^2 + (d_S(s,s') - d_X(x,x'))^2 + (d_S(s,s') - d_2(y,y'))^2.
\end{align*}
We consider the associated multi-marginal GW problem
\begin{equation}\label{eq:MGW-eps}
\MGW_\eps(\SS,\XX,\YY) \coloneqq \inf_{\pi \in \Pi(\sigma,\mu,\nu)} \int_{(S \times X \times Y)^2} c_{\eps}((s,x,y),(s',x',y')) \dx \pi(s,x,y) \dx \pi(s',x',y').
\end{equation}
The following result is analogous to \cite[Prop.~8]{nenna2022transport} 
which provides an analogous characterization in the OT setting.

\begin{theorem}
    \label{thm:MGW-LGW}
    For all $\eps > 0$, 
    let $\pi_\eps \in \Pi(\sigma,\mu,\nu)$ be a solution to $\MGW_\eps(\SS,\XX,\YY)$
    in \cref{eq:MGW-eps}.
    Then every (weak) accumulation point $\bar{\pi}$ of $(\pi_\epsilon)_{\epsilon > 0}$ 
    minimizes the LGW functional in \cref{eq:LGW}, i.e.\
    \[
    \LGW_\SS(\XX,\YY) = \int_{(S \times X \times Y)^2}  \lvert d_X(x,x') - d_Y(y,y') \rvert^2 \dx \bar{\pi}(s,x,y) \dx \bar{\pi}(s',x',y').
    \]
\end{theorem}

\begin{proof}
    Let $\bar{\pi}$ be an accumulation point of $(\pi_\eps)_{\eps>0}$.
    Hence there exists a subsequence $(\pi_{\epsilon(n)})_{n \in \NN}$
    with $\pi_{\eps(n)} \weakly \bar \pi$ where $(\eps(n))_{n \in \NN}$ is a monotonically decreasing sequence with $\eps(n) \to 0$ as $n \to \infty$.
    First, we show that $\bar{\pi}$ is feasible,
    which means
    that the plans $(P_{S\times X})_\# \bar \pi$ and $(P_{S\times Y})_\# \bar \pi$
    are optimal.
    For this, let $\tilde \pi \in \p(S \times X \times Y)$ be an arbitrary plan with $(P_{S \times X})_\# \tilde \pi \in \Pio(\SS,\XX)$ and $(P_{S \times Y})_\# \tilde \pi \in \Pio(\SS,\YY)$.
    For every $\bar{n} \in \NN$,
    the optimality of $\pi_{\eps(\bar n)}$ and the monotonicity of the cost function imply
    \begin{align*}
    &\int_{(S \times X \times Y)^2}
    c_{\eps(\bar n)}((s,x,y),(s',x',y'))
    \dx \tilde \pi(s,x,y) \dx \tilde \pi(s',x',y') 
    \\
    &\geq
    \int_{(S \times X \times Y)^2} 
    c_{\eps(\bar n)}((s,x,y),(s',x',y')) 
    \dx \pi_{\eps(\bar n)}(s,x,y) \dx \pi_{\eps(\bar n)}(s',x',y')
    \\
    &\geq 
    \int_{(S \times X \times Y)^2}
    c_{\eps(n)}((s,x,y),(s',x',y'))
    \dx \pi_{\eps(\bar n)}(s,x,y) \dx \pi_{\eps(\bar n)}(s',x',y')
    \end{align*}
    for all $n \ge \bar n$.
    By the compactness of the spaces,
    the linearity of the integral ensures 
    that the right-hand side converges as $n \to \infty$.
    More precisely,
    we obtain
    \begin{align*}
    &\int_{(S \times X \times Y)^2}
    c_{\eps(\bar n)}((s,x,y),(s',x',y'))
    \dx \tilde \pi (s,x,y) \dx \tilde \pi (s',x',y')
    \\
    &\ge
    \int_{(S \times X \times Y)^2}
    (d_S(s,s') - d_X(x,x'))^2 + (d_S(s,s') - d_Y(y,y'))^2 
    \dx \pi_{\eps(\bar n)}(s,x,y) \dx \pi_{\eps(\bar n)}(s',x',y').
    \end{align*}
    Now passing the limit $\bar n \to \infty$ on both sides,
    and using the linearity of the integral
    as well as the weak convergence of $(\pi_{\eps(\bar n)})_{\bar n \in \NN}$,
    we have
    \begin{align*}
    &\GW^2(\SS,\XX) + \GW^2(\SS,\YY)
    \\
    &= \int_{(S \times X \times Y)^2} (d_S(s,s') - d_X(x,x'))^2 + (d_S(s,s') -d_Y(y,y'))^2 \dx \tilde\pi(s,x,y) \dx \tilde\pi(s',x',y')\\
    &\geq 
    \int_{(S \times X \times Y)^2} 
    (d_S(s,s') -d_X(x,x'))^2 
    +
    (d_S(s,s') - d_Y(y,y'))^2 
    \dx \bar{\pi}(s,x,y)\dx \bar{\pi}(s',x',y')\\
    &=
    \int_{(S \times X \times Y)^2} 
    (d_S(s,s') -d_X(x,x'))^2 
    \dx \bar{\pi}(s,x,y)\dx \bar{\pi}(s',x',y')
    \\
    &\quad+
    \int_{(S \times X \times Y)^2}
    (d_S(s,s') - d_Y(y,y'))^2 
    \dx \bar{\pi}(s,x,y)\dx \bar{\pi}(s',x',y').
    \end{align*}
    Since the first integral on the right-hand side is an upper bound for $\GW^2(\SS,\XX)$,
    and the second for $\GW^2(\SS,\YY)$,
    both integrals have to coincide with $\GW^2(\SS,\XX)$ and $\GW^2(\SS,\YY)$ respectively.
    The marginals of $\bar \pi$ are thus optimal,
    and $\bar \pi$ is feasible.
    
    Next, we show that $\bar{\pi}$ is optimal for $\LGW_\SS(\XX,\YY)$.
    Let $\tilde \pi \in \p(S \times X \times Y)$ be again an arbitrary plan
    with $(P_{S \times X})_\# \tilde \pi \in \Pio(\SS,\XX)$ 
    and $(P_{S \times Y})_\# \tilde \pi \in \Pio(\SS,\YY)$.
    Exploiting the optimality of $(P_{S \times X})_\#\tilde\pi$
    and $(P_{S \times Y})_\#\tilde\pi$,
    we get
    \begin{align*}
        &\int_{(S \times X \times Y)^2}
        \eps(n) \, (d_X(x,x') - d_Y(y,y'))^2 
        \dx \pi_{\eps(n)}(s,x,y) \dx \pi_{\eps(n)}(s',x',y')
        \\
        &\le 
        \int_{(S \times X \times Y)^2}
        \eps(n) \, (d_X(x,x') - d_Y(y,y'))^2
        \dx \pi_{\eps(n)}(s,x,y) \dx \pi_{\eps(n)}(s',x',y')
        \\
        &\quad+
        \int_{(S \times X \times Y)^2}
        (d_S(s,s') - d_X(x,x'))^2 
        + (d_S(s,s') - d_Y(y,y'))^2 
        \dx \pi_{\eps(n)}(s,x,y) \dx \pi_{\eps(n)}(s',x',y')
        \\
        &\quad- 
        \int_{(S \times X \times Y)^2} 
        (d_S(s,s') - d_X(x,x'))^2 
        + (d_S(s,s') -d_Y(y,y'))^2 
        \dx \tilde\pi(s,x,y) \dx \tilde\pi(s',x',y')
        \\
        &=
        \int_{(S \times X \times Y)^2}
        c_{\eps(n)}((s,x,y),(s',x',y'))
        \dx \pi_{\eps(n)}(s,x,y) \dx \pi_{\eps(n)}(s',x',y')\\
        &\quad-
        \int_{(S \times X \times Y)^2}
        (d_S(s,s') - d_X(x,x'))^2 
        + (d_S(s,s') - d_Y(y,y'))^2 
        \dx \tilde\pi(s,x,y) \dx \tilde\pi(s',x',y').
    \end{align*}
    Due to the optimality of $\pi_{\eps(n)}$ 
    with respect to the first term in the last line,
    we continue the estimation by
    \begin{align*}
        &\int_{(S \times X \times Y)^2}
        \eps(n) \, (d_X(x,x') - d_Y(y,y'))^2 
        \dx \pi_{\eps(n)}(s,x,y) \dx \pi_{\eps(n)}(s',x',y')
        \\
        &\leq 
        \int_{(S \times X \times Y)^2} 
        c_{\eps(n)}((s,x,y),(s',x',y'))
        \dx \tilde\pi(s,x,y) \dx \tilde\pi(s',x',y')\\
        &\quad-
        \int_{(S \times X \times Y)^2}
        (d_S(s,s') - d_X(x,x'))^2
        + (d_S(s,s') - d_Y(y,y'))^2 
        \dx \tilde\pi(s,x,y) \dx \tilde\pi(s',x',y')
        \\
        &= 
        \int_{(S \times X \times Y)^2} 
        \eps(n) \, (d_X(x,x') - d_Y(y,y'))^2 
        \dx \tilde \pi(s,x,y) \dx \tilde \pi(s',x',y').
    \end{align*}
    Dividing both sides by $\eps(n)$ yields
    \begin{align*}
    &\int_{(S \times X \times Y)^2} (d_X(x,x') - d_Y(y,y'))^2 
        \dx \pi_{\eps(n)}(s,x,y) \dx \pi_{\eps(n)}(s',x',y')\\
        &\le \int_{(S \times X \times Y)^2} (d_X(x,x') - d_Y(y,y'))^2 
        \dx \tilde \pi(s,x,y) \dx \tilde \pi(s',x',y').
    \end{align*}
    Thus, passing the limit and exploiting the weak convergence of $(\pi_{\eps(n)})_{n \in \NN}$, we arrive at
    \begin{align*}
    &\int_{(S \times X \times Y)^2} (d_X(x,x') - d_Y(y,y'))^2 
        \dx \bar{\pi}(s,x,y) \dx \bar{\pi}(s',x',y')\\
        &\le \int_{(S \times X \times Y)^2} (d_X(x,x') - d_Y(y,y'))^2 
        \dx \tilde \pi(s,x,y) \dx \tilde \pi(s',x',y').
    \end{align*}
    Minimizing the right-hand side over all $\tilde{\pi} \in \p(S \times X \times Y)$ with $(P_{S \times X})_\# \tilde{\pi} \in \Pio(\SS,\XX)$ and $(P_{S \times Y})_\# \tilde{\pi} \in \Pio(\SS,\YY)$ we finally obtain
    \begin{align*}
        &\int_{(S \times X \times Y)^2} (d_X(x,x') - d_Y(y,y'))^2 
        \dx \bar{\pi}(s,x,y) \dx \bar{\pi}(s',x',y')\\
        &\le 
        \inf_{
\substack{\tilde{\pi} \in \p(S \times X \times Y) \\ (P_{S \times X})_\# \tilde{\pi} \in \Pio(\SS,\XX) \\ 
(P_{S \times Y})_\# \tilde{\pi} \in \Pio(\SS,\YY)}}
    \int_{(S \times X \times Y)^2} (d_X(x,x') - d_Y(y,y'))^2 
        \dx \tilde \pi(s,x,y) \dx \tilde \pi (s',x',y')\\[5pt]
        &= \LGW_\SS(\XX,\YY),
    \end{align*}
    yielding the desired result.
\end{proof}

\begin{remark}
    Due to the weak compactness of $\Pi(\sigma,\mu,\nu)$, 
    the sequence $(\pi_\eps)_{\eps > 0}$ in \cref{thm:MGW-LGW}
    has at least one accumulation point;
    so there always exists a weakly convergent sequence $(\pi_{\eps(n)})_{n \in \NN}$
    whose limit yields an optimal plan of the LGW functional in \cref{eq:LGW}.
\end{remark}

\section*{Acknowledgments}
The funding by the German Research Foundation (DFG) within the RTG
2433 DAEDALUS and by the BMBF project ``VI-Screen'' (13N15754) is
gratefully acknowledged.

\vspace{\baselineskip}


\begin{thebibliography}{[10]}

\bibitem{memoli2011gromov}
 \textsc{F.~M{\'e}moli},
 \jr{Found. Comput. Math.} \textbf{11}(4), 417--487 (2011).


\othercit
\bibitem{VVFCC21}
 \textsc{C.~Vincent-Cuaz},  \textsc{T.~Vayer},  \textsc{R.~Flamary},
  \textsc{M.~Corneli},  and  \textsc{N.~Courty},
  Proceedings ICML 2021, Virtual Only
  (PMLR, 2021),  pp.\,10564--10574.


\othercit
\bibitem{pmlr-v162-brogat-motte22a}
 \textsc{L.~Brogat-Motte},  \textsc{R.~Flamary},  \textsc{C.~Brouard},
  \textsc{J.~Rousu},  and  \textsc{F.~D'Alch{\'e}-Buc},
  Proceedings ICML 2022, Baltimore, USA
  (PMLR, 2022),  pp.\,2321--2335.


\bibitem{sliced_gw}
 \textsc{T.~Vayer},  \textsc{R.~Flamary},  \textsc{R.~Tavenard},
  \textsc{L.~Chapel},  and  \textsc{N.~Courty},
 \jr{arXiv:1905.10124} (2019).


\othercit
\bibitem{quantized_gw}
 \textsc{S.~Chowdhury},  \textsc{D.~Miller},  and  \textsc{T.~Needham},
  Proceedings ECML PKDD 2021, Virtual Only
  (Springer, Cham, 2021),  pp.\,811--827.


\bibitem{sampled_gw}
 \textsc{T.~Kerdoncuff},  \textsc{R.~Emonet},  and  \textsc{M.~Sebban},
 \jr{Mach. Learn.} \textbf{110}(8), 2151--2186 (2021).


\bibitem{BBS2021linear}
 \textsc{F.~Beier},  \textsc{R.~Beinert},  and  \textsc{G.~Steidl},
 \jr{arXiv:2112.11964} (2021).


\bibitem{wang2013linear}
 \textsc{W.~Wang},  \textsc{D.~Slep{\v{c}}ev},  \textsc{S.~Basu},
  \textsc{J.\,A. Ozolek},  and  \textsc{G.\,K. Rohde},
 \jr{Int. J. Comput. Vis.} \textbf{101}(2), 254--269 (2013).


\bibitem{WOSSR2011}
 \textsc{W.~Wang},  \textsc{J.\,A. Ozolek},  \textsc{D.~Slep\v{c}ev},
  \textsc{A.\,B. Lee},  \textsc{C.~Chen},  and  \textsc{G.\,K. Rohde},
 \jr{IEEE Trans. Med. Imaging} \textbf{30}(3), 621--631 (2011).


\bibitem{RHNHR2020}
 \textsc{A.\,H.\,M. Rubaiyat},  \textsc{K.\,M. Hallam},  \textsc{J.\,M.
  Nichols},  \textsc{M.\,N. Hutchinson},  \textsc{S.~Li},  and  \textsc{G.\,K.
  Rohde},
 \jr{IEEE Trans. Signal Process.} \textbf{68}(68), 3312--3324 (2020).


\bibitem{KTOR2016}
 \textsc{S.~Kolouri},  \textsc{A.~Tosun},  \textsc{J.~Ozolek},  and
  \textsc{G.~Rohde},
 \jr{Pattern Recognit.} \textbf{51}, 453--462 (2016).


\bibitem{PKKR2017}
 \textsc{S.~Park},  \textsc{S.~Kolouri},  \textsc{S.~Kundu},  and
  \textsc{G.~Rohde},
 \jr{Appl. Comput. Harmon. Anal.} (2017).


\bibitem{moosmueller2021linear}
 \textsc{C.~Moosm\"{u}ller} and  \textsc{A.~Cloninger},
 \jr{arXiv:2008.09165} (2021).


\bibitem{khurana2022supervised}
 \textsc{V.~Khurana},  \textsc{H.~Kannan},  \textsc{A.~Cloninger},  and
  \textsc{C.~Moosm{\"u}ller},
 \jr{arXiv:2201.10590} (2022).


\bibitem{EN2020}
 \textsc{T.\,H. Emerson} and  \textsc{J.\,M. Nichols},
 \jr{Pattern Recognit. Lett.} \textbf{133}, 123--128 (2020).


\bibitem{BKR2014}
 \textsc{S.~Basu},  \textsc{S.~Kolouri},  and  \textsc{G.~Rohde},
 \jr{Proc. Natl. Acad. Sci. USA} \textbf{111}(9), 3448--3453 (2014).


\bibitem{OTWCKBHR2014}
 \textsc{J.~Ozolek},  \textsc{A.~Tosun},  \textsc{W.~Wang},  \textsc{C.~Chen},
  \textsc{S.~Kolouri},  \textsc{S.~Basu},  \textsc{H.~Huang},  and
  \textsc{G.~Rohde},
 \jr{Med. Image. Anal.} \textbf{18}(5), 772--780 (2014).


\bibitem{TYKSR2015}
 \textsc{A.\,B. Tosun},  \textsc{O.~Yergiyev},  \textsc{S.~Kolouri},
  \textsc{J.\,F. Silverman},  and  \textsc{G.\,K. Rohde},
 \jr{Cytometry A} \textbf{87}(4), 326--333 (2015).


\bibitem{PNAS21}
 \textsc{M.~Eckermann},  \textsc{B.~Schmitzer},  \textsc{F.~van\,der Meer},
  \textsc{J.~Franz},  \textsc{O.~Hansen},  \textsc{C.~Stadelmann},  and
  \textsc{T.~Salditt},
 \jr{PNAS} \textbf{118}(48), e2113835118 (2021).


\bibitem{GLDY2019}
 \textsc{S.~Guan},  \textsc{B.~Liao},  \textsc{Y.~Du},  and
  \textsc{X.~Yin},
  Proceedings ICSESS 2019, Beijing, China
  (IEEE, 2019), pp.\,1--4.


\bibitem{PCNWDR2018}
 \textsc{S.\,R. Park},  \textsc{L.~Cattell},  \textsc{J.\,M. Nichols},
  \textsc{A.~Watnik},  \textsc{T.~Doster},  and  \textsc{G.\,K. Rohde},
 \jr{Opt. Express} \textbf{26}(4), 4004--4022 (2018).


\bibitem{BBS2022multi}
 \textsc{F.~Beier},  \textsc{R.~Beinert},  and  \textsc{G.~Steidl},
 \jr{arXiv:2205.06725} (2022).


\bibitem{GS98}
 \textsc{W.~Gangbo} and  \textsc{A.~\'{S}wi\c{e}ch},
 \jr{Comm. Pure Appl. Math.} \textbf{51}(1), 23--45 (1998).


\bibitem{AC2011}
 \textsc{M.~Agueh} and  \textsc{G.~Carlier},
 \jr{SIAM J. Math. Anal.} \textbf{43}(2), 904--924 (2011).


\bibitem{CE10}
 \textsc{G.~Carlier} and  \textsc{I.~Ekeland},
 \jr{Econ. Theory} \textbf{42}(2), 397--418 (2010).


\bibitem{CK18}
 \textsc{Y.~{Chen}} and  \textsc{J.~{Karlsson}},
 \jr{IEEE Contr. Syst. Lett.} \textbf{2}(2), 260--265 (2018).


\othercit
\bibitem{DC14}
 \textsc{M.~Cuturi} and  \textsc{A.~Doucet},
 Proceedings ICML 2014,  Beijing, China (PMLR, 2014),  pp.\,685--693.


\bibitem{EHJK20}
 \textsc{F.~Elvander},  \textsc{I.~Haasler},  \textsc{A.~Jakobsson},  and
  \textsc{J.~Karlsson},
 \jr{Signal Process.} \textbf{171}, 107474 (2020).


\bibitem{sturm2020space}
 \textsc{K.\,T. Sturm},
 \jr{arXiv:1208.0434} (2020).


\othercit
\bibitem{PCS2016}
 \textsc{G.~Peyr\'e},  \textsc{M.~Cuturi},  and  \textsc{J.~Solomon},
 Proceedings ICML 2016,  New York, USA (PMLR, 2016),
  pp.\,2664--2672.


\othercit
\bibitem{SejViaPey21}
 \textsc{T.~S\'ejourn\'e},  \textsc{F.\,X. Vialard},  and  \textsc{G.~Peyr\'e},
 Proceedings NeurIPS 2021,  Virual Only (Curran Associates,
  Inc., 2021), pp.\ 1--14.


\bibitem{BLNS2021}
 \textsc{F.~Beier},  \textsc{J.~von Lindheim},  \textsc{S.~Neumayer},  and
  \textsc{G.~Steidl},
 \jr{arXiv:2103.10854} (2021).


\bibitem{nenna2022transport}
 \textsc{L.~Nenna} and  \textsc{B.~Pass},
 \jr{arXiv:2201.00875} (2022).


\end{thebibliography}

\providecommand{\WileyBibTextsc}{}
\let\textsc\WileyBibTextsc
\providecommand{\othercit}{}
\providecommand{\jr}[1]{#1}
\providecommand{\etal}{~et~al.}

\end{document}